\documentclass[12pt]{amsart}
\usepackage{amssymb,latexsym,graphicx,vaucanson-g}
\usepackage{enumerate,xypic}
\usepackage[top=.8in,bottom=.8in,right=.9in,left=.9in]{geometry}

\makeatletter\@namedef{subjclassname@2010}{%
  \textup{2010} Mathematics Subject Classification}
\makeatother

\newtheorem{thm}[equation]{Theorem}
\newtheorem{cor}[equation]{Corollary}
\newtheorem{prop}[equation]{Proposition}
\newtheorem{lem}[equation]{Lemma}

\numberwithin{equation}{section}

\theoremstyle{definition}

\newtheorem{rem}[equation]{Remark}

\newtheorem{exa}[equation]{Example}

\newcommand{\C}{\ensuremath{{\mathcal{C}}}}
\renewcommand{\S}{\ensuremath{{\mathcal{S}}}}
\renewcommand{\O}{\ensuremath{{\mathcal{O}}}}

\renewcommand{\P}{\ensuremath{{\mathbb{P}}}}
\newcommand{\Q}{\ensuremath{{\mathbb{Q}}}}

\newcommand{\F}{\ensuremath{{\mathbb{F}}}}

\newcommand{\p}{\ensuremath{{\mathfrak{p}}}}
\newcommand{\q}{\ensuremath{{\mathfrak{q}}}}

\DeclareMathOperator{\Div}{Div}

\DeclareMathOperator{\End}{End}

\DeclareMathOperator{\Hom}{Hom}

\DeclareMathOperator{\res}{res}

\DeclareMathOperator{\ord}{ord}

\begin{document}

\title[Automatic Sequences and Curves over Finite Fields]{Automatic Sequences and Curves over Finite Fields}

\author[A. Bridy]{Andrew Bridy}
\address{Andrew Bridy\\Department of Mathematics\\ University of Rochester\\
Rochester, NY 14627, USA}
\email{abridy@ur.rochester.edu}

\date{}

\begin{abstract}
We prove that if $y=\sum_{n=0}^\infty{\bf a}(n)x^n\in\F_q[[x]]$ is an algebraic power series of degree $d$, height $h$, and genus $g$, then the sequence ${\bf a}$ is generated by an automaton with at most $q^{h+d+g-1}$ states, up to a vanishingly small error term. This is a significant improvement on previously known bounds. Our approach follows an idea of David Speyer to connect automata theory with algebraic geometry by representing the transitions in an automaton as twisted Cartier operators on the differentials of a curve.
\end{abstract}



\maketitle

\section{Introduction}

Our starting point is the following well-known theorem of finite automata theory.

\begin{thm}[Christol]\cite{CKMR,Christol}
The power series $y=\sum_{n=0}^\infty {\bf a}(n)x^n\in\F_q[[x]]$ is algebraic over $\F_q(x)$ if and only if the sequence ${\bf a}$ is $q$-automatic.
\end{thm}

Christol's theorem establishes a dictionary between automatic sequences and number theory in positive characteristic. The purpose of this paper is to investigate how complexity translates across this dictionary. The secondary purpose is to demonstrate an intimate connection between automatic sequences and the algebraic geometry of curves.

We take the complexity of a sequence {\bf a} to be \emph{state complexity}. Let $N_q({\bf a})$ denote the number of states in a minimal $q$-automaton that generates ${\bf a}$ in the reverse-reading convention (this will be defined precisely in Section 2). If $y$ is algebraic over $k(x)$, let the \emph{degree} $\deg(y)$ be the usual field degree $[k(x)[y]:k(x)]$ and the \emph{height} $h(y)$ be the minimal $x$-degree of a bivariate polynomial $f(x,T)\in k[x,T]$ such that $f(x,y)=0$. The \emph{genus} of $y$ will be the genus of the normalization of the projective closure of the affine plane curve defined by the minimal polynomial of $y$.

We bound the complexity of ${\bf a}$ in terms of the degree, height, and genus of $y$. In Section \ref{Background} we review some known lower bounds. Our main result is the following upper bound.

\begin{thm}\label{thm: main}
Let $y=\sum_{n=0}^\infty {\bf a}(n)x^n\in\F_q[[x]]$ be algebraic over $\F_q(x)$ of degree $d$, height $h$, and genus $g$. Then $$N_q({\bf a})\leq (1+o(1))q^{h+d+g-1}.$$ The $o(1)$ term tends to 0 for large values of any of $q$, $h$, $d$, or $g$.
\end{thm}

All previous upper bounds are much larger. These are usually stated in terms of the $q$-kernel of ${\bf a}$, which can be described as the orbit of a certain semigroup acting on the series $y$ and is in bijection with a minimal automaton that outputs ${\bf a}$ (see Theorem \ref{Eilenberg}). The best previous bound is due to Fresnel, Koskas, and de Mathan, who show that
\begin{equation}\label{best previous}
N_q({\bf a})\leq q^{qd(h(2d^2-2d+1)+C)}
\end{equation}
for some $C=C(q)$ that they do not seem to compute exactly \cite[Thm 2.2]{FresnelKoskas}. Adamczewski and Bell prove a bound that is roughly $q^{d^4 h^2 p^{5d}}$ where $p=\text{char }\F_q$ \cite[p. 383]{ABVanishing}. Earlier, Derksen showed a special case of the bound of Adamczewski and Bell for rational functions \cite[Prop 6.5]{Derksen}. Harase proved a larger bound using essentially the same technique \cite{Harase1,Harase2}. It should be noted that some of these results hold in more generality than the setting of this paper: the techniques of Adamczewski, Bell, and Derksen apply to power series in several variables over infinite ground fields of positive characteristic (with appropriate modifications).

In Proposition \ref{rational sharp} we show that Theorem \ref{thm: main} is qualitatively sharp for the power series expansions of rational functions, in that it is sharp if we replace the $o(1)$ term by 0. There are some special cases where the bound can be improved. For example, an easy variation of our main argument shows that if $y\,dx$ is a holomorphic differential on the curve defined by the minimal polynomial of $y$, then $N_q({\bf a})\leq q^{d+g-1}$ (see Example \ref{elliptic example}). It is also possible to give a coarser estimate that is independent of $g$, which shows that Theorem \ref{thm: main} compares favorably to the work of Fresnel et al. even when the genus is as large as possible:
\begin{cor}
Under the hypotheses of Theorem \ref{thm: main}, $N_q({\bf a})\leq (1+o(1))q^{hd}$.
\end{cor}
\begin{proof}
Let $X$ be the curve defined by the minimal polynomial of $y$. Observe that $d$ is the degree of the map $\pi_x:X\to\P^1$ that projects on the $x$-coordinate. Likewise, $h$ is the degree of the projection map $\pi_y:X\to \P^1$. Therefore $g\leq (d-1)(h-1)$ by Castelnuovo's inequality (actually, a special case originally due to Riemann, see \cite[Cor 3.11.4]{Stichtenoth}).
\end{proof}

Though the reverse-reading convention is the natural one to use in the context of algebraic series, our approach also gives an upper bound on the state complexity of forward-reading automata via a dualizing argument. Let $N_q^f({\bf a})$ denote the minimal number of states in a forward-reading automaton that generates the sequence ${\bf a}$, and let $y$ be as in Theorem \ref{thm: main}. Then we have the following.
\begin{thm}\label{thm: main2}
$N_q^f({\bf a})\leq q^{h+2d+g-1}$.
\end{thm}
To our knowledge, there are no previous bounds on foward-reading complexity in this context except for the well known observation that $N_q^f({\bf a})\leq q^{N_q({\bf a})}$ (see Proposition \ref{dimension bound}).

The key idea in our argument is to recast finite automata in the setting of algebraic geometry. If $y=\sum_{n=0}^\infty {\bf a}(n) x^n$ is algebraic, then it lies in the function field of a curve $X$, and a minimal reverse-reading automaton that generates ${\bf a}$ embeds into the differentials of $X$ in a natural way. This brings to bear the machinery of algebraic curves, and in particular the Riemann-Roch theorem. This idea was introduced by David Speyer in \cite{Speyer} and used to give a new proof of the ``algebraic implies automatic'' direction of Christol's theorem. Building on Speyer's work, we improve the complexity bound implicit in his proof.

The paper is organized as follows. Section \ref{Background} reviews the theory of finite automata and automatic sequences. Section \ref{Curves} introduces the connection with algebraic geometry, leading to the proof of Theorem \ref{thm: main}. We also discuss the problem of state complexity growth as the field varies: if $K$ is a number field and $y=\sum_{n=0}^\infty{\bf a}(n)x^n\in K[[x]]$ is algebraic over $K(x)$, then the state complexity of the reduced sequences ${\bf a}_\p$ varies with the prime $\p$ in a way controlled by the algebraic nature of $y$. Section \ref{Examples} illustrates some detailed examples of our method.

\subsection*{Acknowledgements}
The author would like to thank David Speyer for first exploring this striking link between automata and geometry and for allowing an exposition of his work in \cite{Speyer}. The author also thanks Eric Bach and Jeffrey Shallit for many helpful conversations related to the topics in the paper, and the referee for a careful reading of the paper and helpful comments. Diagrams of automata were produced using the VauCanSon-G Latex package of Sylvain Lombardy and Jacques Sakarovitch.

\section{Automata, Sequences, and Representations}\label{Background}

\subsection{Finite Automata and Automatic Sequences}
A comprehensive introduction to finite automata and automatic sequences can be found in the book of Allouche and Shallit \cite{AS}. We give a brief overview of the theory for the convenience of the reader.

A \emph{finite automaton} or \emph{DFAO} (\emph{Deterministic Finite Automaton with Output}) $M$ consists of a finite set $\Sigma$ known as the \emph{input alphabet}, a finite set $\Delta$ known as the \emph{output alphabet}, a finite set of \emph{states} $Q$, a distinguished \emph{initial state} $q_0\in Q$, a \emph{transition function} $\delta:Q\times\Sigma\to Q$, and an \emph{output function} $\tau:Q\to\Delta$. 

Let $\Sigma^*$ (the \emph{Kleene closure} of $\Sigma$) be the monoid of all finite-length words over $\Sigma$ under the operation of juxtaposition, including an empty word as the identity element. The function $\delta$ can be prolonged to a function $\delta:Q\times\Sigma^*\to Q$ by inductively defining $\delta(q_i,wa)=\delta(\delta(q_i,w),a)$ for $w\in\Sigma^*$ and $a\in\Sigma$. Therefore a DFAO $M$ induces a map $f_M:\Sigma^*\to \Delta$ defined by $f_M(w)=\tau(\delta(q_0,w))$ under the \emph{forward-reading} convention. If we let $w^R$ denote the reverse of the word $w$, then the \emph{reverse-reading} convention is $f_M(w)=\tau(\delta(q_0,w^R))$. A function $f:\Sigma^*\to\Delta$ is a \emph{finite-state function} if $f=f_M$ for some DFAO $M$. A DFAO is \emph{minimal} if it has the smallest number of states among automata that induce the same function $f_M$.

A helpful way of visualizing a DFAO $M$ is through its \emph{transition diagram}. This is a directed graph with vertex set $Q$ and directed edges that join $q$ to $\delta(q,a)$ for each $q\in Q$ and $a\in\Sigma$. The initial state is marked by an incoming arrow with no source. The states are labeled by their output $\tau(q)$. Figure \ref{Thue-Morse} shows the transition diagram of the Thue-Morse DFAO $T$ with $\Sigma=\Delta=\{0,1\}$, where $f_T(w)=1$ if and only if $w\in\{0,1\}^*$ contains an odd number of $1$s.

\begin{figure}[ht]
\caption{Thue-Morse 2-DFAO $T$}
\includegraphics[scale=.75]{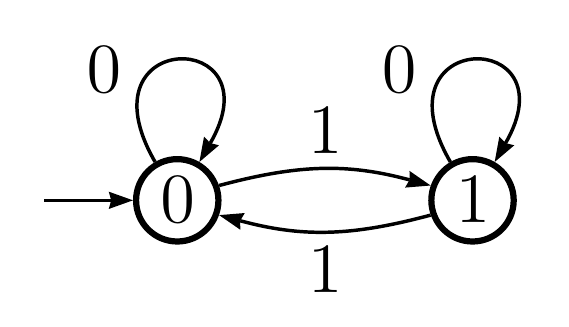}
\label{Thue-Morse}
\end{figure}

Let $p\geq 2$ be an integer (not necessarily prime) and let $(n)_p$ denote the base-$p$ expansion of the integer $n\geq 0$. A sequence ${\bf a}$ is \emph{$p$-automatic} if there exists a DFAO $M$ with input alphabet $\Sigma_p=\{0,1,\dots,p-1\}$ such that ${\bf a}(n)=f_M((n)_p)$; we say that $M$ \emph{generates} ${\bf a}$. It is known that a sequence is $p$-automatic with respect to the forward-reading convention if and only if it is $p$-automatic with respect to the reverse-reading convention \cite[Thm 5.2.3]{AS}. (We prove a quantitative version of this fact in Proposition \ref{dimension bound}.) A DFAO with input alphabet $\Sigma_p$ is called a $p$-DFAO. 

Let ${\bf a}$ be a $p$-automatic sequence. As in the introduction, the \emph{forward-reading complexity} $N_p^f({\bf a})$ is the number of states in a minimal forward-reading $p$-DFAO that generates ${\bf a}$, and the \emph{reverse-reading complexity} $N_p({\bf a})$ is the number of states in a minimal reverse-reading $p$-DFAO that generates ${\bf a}$.

\begin{rem}\label{Minimality}
Base-$p$ expansions are only unique if we disallow leading zeros -- for example, the binary strings $11$ and $011$ both represent the integer $3$. This creates a minor ambiguity in the minimality of a generating DFAO for a sequence, as it may be the case that a larger DFAO is needed if we require that the same output is produced for \emph{every possible} base-$p$ expansion of an integer (it is not even \emph{a priori} clear that both definitions of ``$p$-automatic" are equivalent; see \cite[Thm 5.2.1]{AS}). Throughout this paper we enforce the stricter requirement that the generating DFAO gives the same output regardless of leading zeros. (This is necessary for minimality to translate correctly from automata to curves.)
\end{rem}

The canonical example of an automatic sequence is the 2-automatic Thue-Morse sequence:
$${\bf a}=01101001\cdots$$
where ${\bf a}(n)=f_T((n)_2)$ for the Thue-Morse automaton $T$. The term ${\bf a}(n)$ is the parity of the sum of the bits in the binary expansion of $n$. Note that $T$ generates ${\bf a}$ under both the forward-reading and reverse-reading conventions, and $T$ is obviously minimal, so $N_2({\bf a})=N_2^f({\bf a})=2$.

There are many characterizations of a sequence that are equivalent to being $p$-automatic. We mention two, which are relevant to computing state complexity. The first is due to Eilenberg and relies on the notion of the $p$-kernel of ${\bf a}$, which is defined to be the set of sequences $n\mapsto{\bf a}(p^in+j)$ for all $i\geq 0$ and $0\leq j\leq p^i-1$. The second actually holds in more generality than automatic sequences: it is an easy adaptation of the Myhill-Nerode theorem to the DFAO model.
\begin{thm}[Eilenberg]\label{Eilenberg}
For $p\geq 2$, the $p$-kernel of ${\bf a}$ is finite if and only if ${\bf a}$ is $p$-automatic. Moreover, $N_p({\bf a})$ is precisely the size of the $p$-kernel of ${\bf a}$.
\end{thm}
\begin{proof}
See \cite[Prop V.3.3]{Eilenberg1974} or \cite[Prop 6.6.2]{AS}. See \cite[Prop 4.9]{Derksen} for the claim of minimality (this minimality is in the strict sense of Remark \ref{Minimality}, as a smaller DFAO may exist otherwise).
\end{proof}
Let $f:\Sigma^*\to\Delta$ be any function. For $x,y\in\Sigma^*$, define $x\sim y$ to mean $f(xz)=f(yz)$ for all $z\in\Sigma^*$. Then $\sim$ is an equivalence relation on $\Sigma^*$, called the Myhill-Nerode equivalence relation.
\begin{thm}[Myhill-Nerode]\label{Myhill-Nerode}
The equivalence relation $\sim$ has finitely many equivalence classes if and only if $f$ is a finite-state function. The number of equivalence classes of $\sim$ is the minimal number of states in a forward-reading DFAO $M$ such that $f=f_M$.
\end{thm}
\begin{proof}
See \cite[Thm 4.1.8]{AS} and \cite[p.\ 149]{AS}. 
\end{proof}

\subsection{Automata from $p$-Representations}
Another way of characterizing $p$-automatic sequences is by $p$-representations. Let ${\bf a}$ be a sequence taking values in a field $k$. A $p$-\emph{representation} of ${\bf a}$ consists of a finite-dimensional vector space $V$ over $k$, a vector $v\in V$, a morphism of monoids $\phi:\Sigma_p^*\to\End(V)$, and a linear functional $\lambda\in V^*=\Hom(V,k)$, such that for any positive integer $n$,
$$
{\bf a}(n) = \lambda\phi((n)_p)v.
$$
A sequence that admits a $p$-representation is known as \emph{$p$-regular} \cite[Chapter 16]{AS}. Equivalently, its associated power series is \emph{recognizable} in the language of \cite{BR} (see also \cite{Salomaa}).

We also make the nonstandard but natural definition of a $p$-\emph{antirepresentation} of ${\bf a}$, which consists of the data of a representation except that $\phi:\Sigma_p^*\to \End(V)$ is an antimorphism of monoids, that is, $\phi(wv)=\phi(v)\phi(w)$. Antirepresentations on $V$ correspond to representations on the dual space $V^*$. In Proposition \ref{dimension bound} we show that when $k$ is finite of characteristic $p$, $p$-representations give rise to reverse-reading automata and $p$-antirepresentations give rise to forward-reading automata. 

An obvious necessary condition for a sequence to be automatic is that it assumes finitely many values. It is not hard to show that a sequence over a field is $p$-automatic if and only if it is both $p$-regular and assumes finitely many values (see \cite[Thm 16.1.5]{AS} or \cite[Thm V.2.2]{BR}). We give a quantitative proof of this fact when $k$ is a finite field, in which case the ``finitely many values" hypothesis holds trivially. Our argument will allow us to deduce a bound on state complexity.
\begin{prop}\label{dimension bound}
Let $p$ be a prime or prime power and let $k=\F_p$. The sequence ${\bf a}$ over $k$ is $p$-regular if and only if it is $p$-automatic. Furthermore, if ${\bf a}$ has a $p$-representation on a vector space $V$ of dimension $m$, then $N_p^f({\bf a})$ and $N_p({\bf a})$ are both at most $p^m$.
\end{prop}
\begin{proof}

First assume that ${\bf a}$ is $p$-automatic. There exists a reverse-reading $p$-DFAO $M$ such that $f_M((n)_p^R)={\bf a}(n)$. We construct a representation for ${\bf a}$ analogous to the regular representation in group theory. Let $q_1,\dots,q_m$ be the states of $M$ and let $V=k^m$. Let $v=e_1$, the first standard basis vector of $V$. For $i\in\Sigma_p$, define the matrix $\phi(i)\in k^{m\times m}\simeq \End(V)$ by
$$
   \phi(i)_{a,b}=\left\{
     \begin{array}{ll}
       1, & \text{if }\delta_i(q_b)=q_a; \\
       0, & \text{otherwise;}
     \end{array}
   \right.
$$
and extend $\phi$ to a morphism from $\Sigma_p^*$ to $k^{m\times m}$. Let $\lambda$ be defined by $\lambda(e_j)=\tau(q_j)$ for each $j$ and extended linearly to a functional $\lambda:V\to k$. This defines a $p$-representation of ${\bf a}$, which can be pictured as embedding the states of $M$ into $V$ and realizing the transition function as a set of $p$ linear transformations.

Now assume instead that ${\bf a}$ is $p$-regular. Let $(V,v,\phi,\lambda)$ be a $p$-representation of ${\bf a}$ with $\dim V=m$. We construct a reverse-reading DFAO $M$ as follows. The initial state $q_0$ is $v$, the set of states is $Q=\{\phi(w)v : w\in\Sigma_p^*\}$, the transition function is given by $\delta(w,i)=\phi(i)(w)$, and the output function is $\tau(w)=\lambda(w)$. It is a matter of unraveling notation to see that $M$ outputs the sequence ${\bf a}$, and $M$ has at most $|V|=p^m$ states.

We now construct a $p$-antirepresentation for the sequence ${\bf a}$. Let $\phi^T:\Sigma_p^*\to\End(V^*)$ be the antimorphism defined by $\phi^T(w)=\phi(w)^T$, where $T$ denotes transpose. Now $(V^*,\lambda,\phi^T,v)$ is a $p$-antirepresentation of ${\bf a}$, where we identify $V$ with $(V^*)^*$ in the natural way. If $(n)_p=c_u\dots c_1c_0$, then
$${\bf a}(n)=\lambda\phi((n)_p)v=\lambda\phi(c_u)\cdots\phi(c_0)v,$$
so thinking of $v$ as an element of $(V^*)^*$, we have
$$ {\bf a}(n)=v(\lambda\phi(c_u)\cdots\phi(c_0))=v\phi^T(c_0)\cdots\phi^T(c_u)\lambda=v\phi^T(c_u\cdots c_0)\lambda = v\phi^T((n)_p)\lambda.$$
This corresponds to a forward-reading DFAO $M$ in the following way: let the initial state $q_0$ of $M$ be $\lambda$, the set of states be $Q=\{\phi^T(w)\lambda:w\in\Sigma_p^*\}$, the transition function be $\delta(\mu,i)=\phi^T(i)(\mu)$, and the output function be $\tau(\mu)=v^T(\mu)=\mu(v)$. Taking transposes has the effect of reversing input words, so this gives a forward-reading DFAO that outputs the sequence ${\bf a}$, and it has at most $p^m$ states, as $\dim V=\dim V^*$.
\end{proof}

\begin{rem}
A special case of the antirepresentation constructed in Proposition \ref{dimension bound} gives a standard result of automata theory: if $\Delta=\{0,1\}$, so that $M$ either accepts or rejects each input string, and $M$ has $n$ states, then a minimal reversed automaton for $M$ has at most $2^n$ states. We can identify $\Delta$ with $\F_2$, and the regular representation is on the vector space $\F_2^n$.
\end{rem}

The representation constructed in Proposition \ref{dimension bound} produces an $p$-DFAO where the states are identified with a subset of $V$ and the transitions are realized as linear transformations. In general, this is not a minimal DFAO. Much of our work in the rest of the paper will be describing canonical representations that produce minimal DFAOs. 

Somewhat surprisingly, for any $p$-representation of ${\bf a}$, the forward-reading $p$-DFAO produced by the antirepresentation in Proposition \ref{dimension bound} is minimal as long as $V$ equals the linear span of $\{\phi(w)v:w\in\Sigma_p^*\}$. This assumption on $V$ loses no generality, because we can always replace $V$ with this subspace (in particular, satisfying this assumption does \emph{not} mean the corresponding reverse-reading automaton is minimal). This observation is to our knowledge new, and we prove it in Proposition \ref{forward minimal} below. In a sense, this is an analogue via representations of the minimization algorithm of Brzozowski \cite{Brzozowski,Shallit2}.

\begin{prop}\label{forward minimal}
Let ${\bf a}$ be a sequence taking values in a finite field $k$, and let $(V,v,\phi,\lambda)$ be a $p$-representation of ${\bf a}$. Assume without loss of generality that $V$ is the $k$-linear span of $\{\phi(w)v:w\in\Sigma_p^*\}$. The DFAO $M$ corresponding to the antirepresentation in Proposition \ref{dimension bound} is a minimal forward-reading $p$-DFAO that generates ${\bf a}$.
\end{prop}
\begin{proof}
The state set of $M$ is $Q=\{\phi^T(w)\lambda:w\in\Sigma_p^* \}$ with initial state $\lambda$, the transition function is $\delta(\mu,i)=\phi(i)^T(\mu)$, and the output function is $\tau(\mu)=\mu(v)$ for some fixed $v\in V$. We show that the states of $Q$ are in one-to-one correspondence with the Myhill-Nerode equivalence classes of the finite-state function $f_M$ as in Theorem \ref{Myhill-Nerode}.

Let $[x]$ be the equivalence class of $x\in\Sigma_p^*$. We need to show that $[x]=[y]$ if and only if $\phi^T(x)(\lambda)=\phi^T(y)(\lambda)$. We have
\begin{equation*}
[x] = \{y\in\Sigma_p^*:\tau(\phi^T(xz)\lambda)=\tau(\phi^T(yz)\lambda)\text{ for all }z\in\Sigma_p^*\} 
\end{equation*}
and the computation 
$$\tau(\phi^T(xz)\lambda)=\tau(\lambda \phi(xz))=\lambda\phi(x)\phi(z)v$$ shows that
\begin{align*}
[x] & = \{y\in\Sigma_p^*: \lambda\phi(x)\phi(z)v=\lambda\phi(y)\phi(z)v\text{ for all }z\in\Sigma_p^*\}\\
& = \{y\in\Sigma_p^*: \lambda\phi(x)=\lambda\phi(y)\text{ in }V^*\}
\end{align*}
because $V$ is the span of the set $\{\phi(z)v:z\in\Sigma_p^*\}$. So $[x]$ is the precisely the set of all $y$ such that $\phi^T(x)\lambda=\phi^T(y)\lambda$. By the Myhill-Nerode theorem, $M$ is a minimal forward-reading DFAO that generates ${\bf a}$.
\end{proof}

\subsection{Power Series and Bounds on Degree and Height}
We develop some standard machinery that is used in the proof of Christol's theorem. Let $k$ be a perfect field of characteristic $p$, for example, a finite field $\F_{p^r}$. Let $y=\sum_{n=-\infty}^\infty {\bf a}(n) x^n\in k((x))$, where ${\bf a}(n)=0$ for all sufficiently large negative $n$. Define
\begin{equation}
\Lambda_i(y)=\sum_{n=-\infty}^\infty {\bf a}(pn+i)^{1/p}x^n.
\end{equation}
The operators $\Lambda_i$ are $\F_p$-linear (not necessarily $k$-linear) endomorphisms of the field $k((x))$. They are known in this context as Cartier operators. Observe that 
\begin{align}
y & = \sum_{i=0}^{p-1}\sum_{n=-\infty}^\infty {\bf a}(pn+i)x^{pn+i} = \sum_{i=0}^{p-1}x^i\sum_{n=-\infty}^\infty  {\bf a}(pn+i)x^{pn} = \sum_{i=0}^{p-1}x^i\left(\sum_{n=-\infty}^\infty {\bf a}(pn+i)^{1/p}x^{n}\right)^p
\end{align}
and therefore
\begin{equation}\label{eqn: power series decimation}
y = \sum_{i=0}^{p-1} x^i (\Lambda_i(y))^p.
\end{equation}

If $y=\sum_{n=0}^\infty {\bf a}(n)x^n\in\F_p[[x]]$, it is easy to see that the $p$-kernel of ${\bf a}$ is in bijection with the orbit of $y$ under the monoid generated by the $\Lambda_i$ operators, as taking $p$th roots fixes each element of $\F_p$. If $y=\sum_{n=0}^\infty {\bf a}(n)x^n\in\F_q[[x]]$ for $q=p^r$, then applying $r$-fold compositions of the $\Lambda_i$ operators gives $q$-ary decimations of the sequence ${\bf a}$. That is, if $$c=i_r p^{r-1}+i_{r-1}p^{r-2}+\dots+i_2p+i_1,$$ where each $i_j\in\{0,\dots,p-1\}$, then
\begin{equation}
\Lambda_{i_1}\Lambda_{i_2}\cdots\Lambda_{i_r}(y)=\sum_{n=0}^\infty  {\bf a}(p^rn+i_rp^{r-1}+\dots+i_2p+i_1)^{1/p^r}x^{n}=\sum_{n=0}^\infty {\bf a}(qn+c)x^{n}
\end{equation}
because $\F_q$ is fixed under taking $q$th roots. It follows that
\begin{equation}\label{eqn: q power series decimation}
y = \sum_{0\leq i_1,\dots,i_r\leq p-1}x^{i_rp^{r-1}+\dots+i_2p+i_1}(\Lambda_{i_1}\Lambda_{i_2}\dots\Lambda_{i_r}(y))^q.
\end{equation}

\begin{rem}
If $k=\F_q$, the operators $\Lambda_i$ are usually defined by $\Lambda_i(y)=\sum_{n=-\infty}^\infty  {\bf a}(qn+i)x^n$, for example in \cite{ABVanishing,ABDiagonal,AS}. With this definition, Equation \ref{eqn: q power series decimation} takes on the much simpler form $y = \sum_{i=0}^{q-1} x^i (\Lambda_i(y))^q$. However, our definition fits more naturally into the geometric setting of Section 3 because it is invariant under base extension, whereas the usual definition depends on a choice of $\F_q$ fixed in advance.
 \end{rem}

Continue to assume that $y\in\F_q[[x]]$, where $p$ is prime and $q=p^r$. Define $S_q$ to be the monoid generated by all $r$-fold compositions of the $\Lambda_i$ operators. The $q$-kernel of ${\bf a}$ is in bijection with the orbit of $y$ under $S_q$, which we denote $S_q(y)$. If ${\bf a}$ is a $q$-automatic sequence, then by Eilenberg's Theorem $S_q(y)$ is finite and $|S_q(y)|=N_q({\bf a})$. Moreover, we have a $q$-representation for ${\bf a}$: $V$ is the finite-dimensional $\F_q$-subspace of $\F_q[[x]]$ spanned by the power series whose coefficient sequences are in the $q$-kernel of ${\bf a}$, $\phi$ is defined so that for $i\in\Sigma_q$, $\phi(i)$ maps $\sum_{n=0}^\infty  {\bf a}(n)x^n$ to $\sum_{n=0}^\infty {\bf a}(qn+i)x^n$ by the $r$-fold composition of the $\Lambda_i$ operators given in Equation \ref{eqn: q power series decimation}, and the linear functional $\lambda$ maps a power series to its constant term. 

The standard proof of the ``algebraic implies automatic" half of Christol's theorem (\cite{CKMR,Christol}; \cite[Thm 12.2.5]{AS}) follows from the observation that $y$ is algebraic if and only if it lies in a finite-dimensional $\F_q$-subspace of $\F_q((x))$ invariant under $S_q$. Given an algebraic power series $y$, it is easy to construct an invariant space using Ore's lemma \cite[pp.\ 355--356]{AS}, which leads to the prior bounds on state complexity mentioned in the introduction, but the dimension of the space constructed is often far larger than the dimension of the linear span of $S_q(y)$. We achieve a sharper bound on the dimension by introducing some relevant machinery from algebraic geometry in the next section. First we demonstrate some easy upper bounds on height and degree in terms of reverse-reading state complexity that can be extracted from the usual proof of Christol's theorem.

\begin{prop}\label{coarse degree bound}
Let $y=\sum_{n=0}^\infty {\bf a}(n) x^n\in\F_q[[x]]$. Assume ${\bf a}$ is $q$-automatic and $N_q({\bf a})=m$. Then $y$ is algebraic, $\deg (y)\leq q^m-1$, and $h(y)\leq mq^{m+1}$. 
\end{prop}
\begin{proof}
Let $S_q(y)=\{y_1,\dots,y_m\}$. From Equation \ref{eqn: q power series decimation}, for each $i\in\{1,\dots,m\}$ we have
$$y_i\in\langle y_1^q,\dots,y_m^q\rangle,$$
where the angle brackets $\langle\dots\rangle$ indicate $\F_q(x)$-linear span. So
$$y_i^q\in\langle y_1^{q^2},\dots,y_m^{q^2}\rangle,$$
and eventually 
$$y_i^{q^m}\in\langle y_1^{q^{m+1}},\dots,y_m^{q^{m+1}}\rangle.$$
Therefore
$$\{y_i,y_i^q,y_i^{q^2},\dots,y_i^{q^m}\}\subseteq \langle y_1^{q^{m+1}},\dots,y_m^{q^{m+1}}\rangle,$$
which forces an $\F_q(x)$-linear relation among $y_i,y_i^q,y_i^{q^2},\dots,y_i^{q^m}$, that is, an algebraic equation satisfied by $y_i$. In particular, $y$ is algebraic, which proves the ``automatic implies algebraic" direction of Christol's theorem. If $y\neq 0$ we can cancel $y$ to deduce $\deg y\leq q^m-1$ (if $y=0$ this is trivially true).

Working through the chain of linear dependences shows that each $y_i^{q^k}$ can be written as a linear combination of $\{y_1^{q^{m+1}},\dots,y_m^{q^{m+1}}\}$ with polynomial coefficients of degree at most $q^{m+1}$. A standard argument in linear algebra shows that there is a vanishing linear combination of $\{y_i,y_i^q,\dots,y_i^{q^m}\}$ with polynomial coefficients of degree at most $mq^{m+1}$.
\end{proof}

It is easy to construct infinite families of power series for which degree and height grow exponentially in $N_q({\bf a})$, which we do in Examples \ref{height example} and \ref{degree example}. It is not clear whether the bounds of Proposition \ref{coarse degree bound} are sharper than these families indicate.

\begin{exa}\label{height example}

Let $y=x^n$ and let ${\bf a}$ be the sequence with a $1$ in the $n$th position and $0$ in every other position.  We have $N_q({\bf a})=\lceil\log_q(n)\rceil + 1$, because a $q$-DFAO generating ${\bf a}$ needs $\lceil\log_q(n)\rceil$ states to recognize the base-$q$ expansion of $n$ and one additional ``trap state" that outputs zero on any input that deviates from this expansion. So $h(y)$ grows exponentially in the number of states required. For example, Figure \ref{One word} gives a 3-DFAO that outputs 1 on the word 201 and 0 otherwise.
\begin{figure}[ht]
\caption{Minimal 3-DFAO that outputs 1 on 201 and 0 otherwise}
\includegraphics[scale=.75]{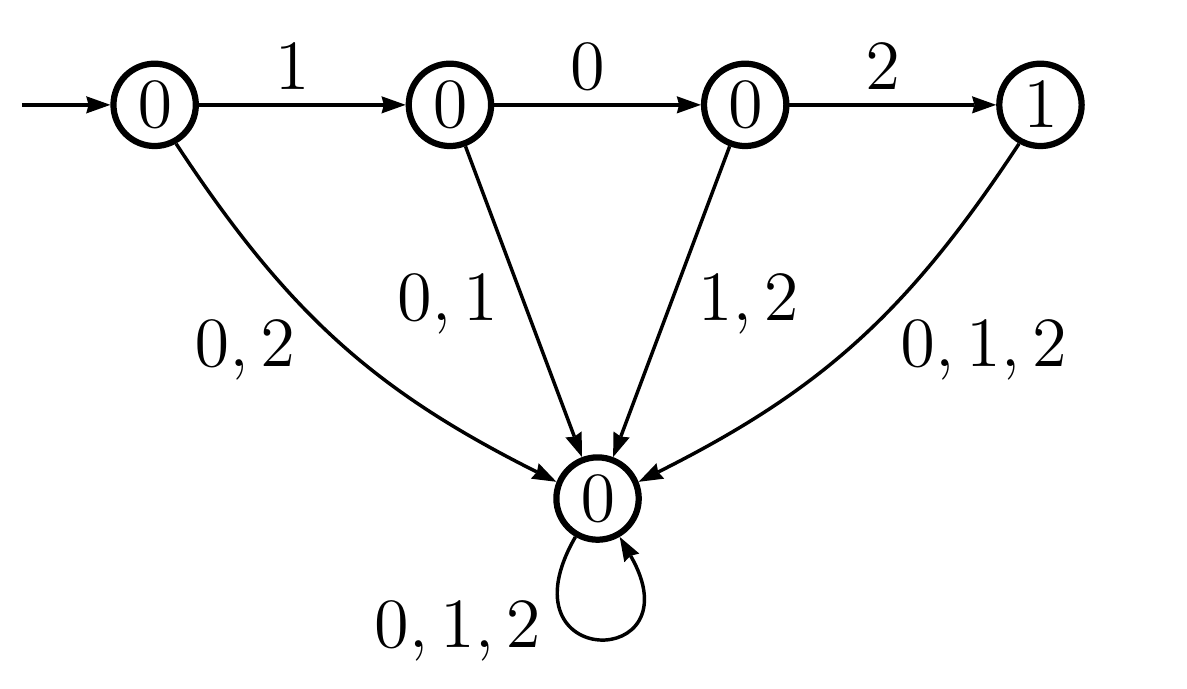}
\label{One word}
\end{figure}
\end{exa}

\begin{exa}\label{degree example}

The degree bound of Proposition \ref{coarse degree bound} is nearly sharp for those degrees that are powers of $q$. We argue that the unique solution in $\F_q[[x]]$ to the Artin-Schreier equation
\begin{equation*}
y^{q^m}-y =x,
\end{equation*}
which is
\begin{equation*}
y = x + x^{q^m} + x^{q^{2m}} + x^{q^{3m}} + \cdots,
\end{equation*}
satisfies $N_q({\bf a})=m+2$, so $\deg(y)=q^{N_q({\bf a})-2}$. As usual, we identify the states of a reverse-reading $q$-automaton that outputs ${\bf a}$ with the orbit of $y$ under $S_q$. We compute
\begin{align*}
\phi(0)(y) & =\sum_{n=0}^\infty {\bf a}(qn)x^n = x^{q^{m-1}} + x^{q^{2m-1}}+x^{q^{3m-1}}+\cdots,\\
\phi(1)(y) & =  \sum_{n=0}^\infty {\bf a}(qn+1)x^n = 1,
\end{align*}
and for $2\leq c\leq q-1$, $\phi(c)(y)=\sum_{n=0}^\infty {\bf a}(qn+c)x^n=0$. It is clear that $S_q(1)=\{0,1\}$, and
\begin{align*}
\phi(0)^2(y) & = x^{q^{m-2}} + x^{q^{2m-2}}+x^{q^{3m-2}}+\cdots,\\
\phi(0)^3(y) & = x^{q^{m-3}} + x^{q^{2m-3}}+x^{q^{3m-3}}+\cdots,\\
& \vdots \\
\phi(0)^{m-1}(y) & =  x^{q} + x^{q^{m+1}}+x^{q^{2m+1}}+\cdots,\\
\phi(0)^m(y) & =  x + x^{q^{m}}+x^{q^{2m}}+\cdots=y.
\end{align*}
Except for $y$, the power series in this list are all $q$th powers, so any element of $S_q$ that includes a $\Lambda_i$ operator other than $\Lambda_0$ sends each one to zero. By Eilenberg's theorem, a minimal reverse-reading automaton that outputs ${\bf a}$ has $m+2$ states. Figure \ref{fig: degree example} depicts such an automaton for $m=4$. Any undrawn transition arrow leads to a trap state $q_T$ (not pictured) where $\delta(q_T,i)=q_T$ for every $i\in\Sigma_q$ and $\tau(q_T)=0$.

\begin{figure}[ht]
\caption{Minimal $q$-DFAO generating the coefficients of $y$, where $y^{q^4}-y=x$}
\includegraphics[scale=.75]{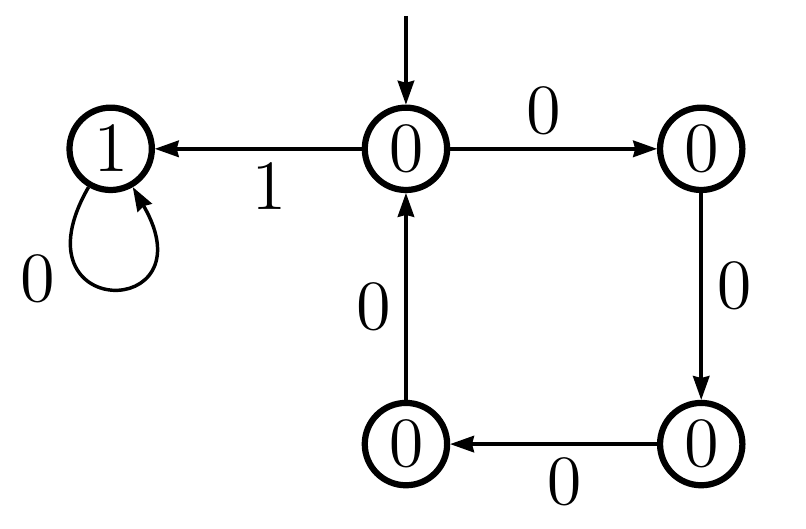}
\label{fig: degree example}
\end{figure}

\end{exa}

A sequence is $p$-automatic if and only if it is $p^r$-automatic for any $r\geq 1$ \cite[Thm 6.6.4]{AS}, so it makes sense to discuss the base-$p$ state complexity of ${\bf a}$, as well as the base-$q$ state complexity. In fact, Christol's theorem is usually stated in the equivalent form that for any $r\geq 1$, a power series over $\F_{p^r}$ is algebraic if and only if its coefficient sequence is $p$-automatic. The next proposition shows that, for a given $q$ and $p$, there is no qualitative difference between base-$p$ and base-$q$ complexity, in the sense they are at most a multiplicative constant apart.

\begin{prop}\label{p vs q}
Let ${\bf a}$ be $p$-automatic and $q=p^r$. Then $$N_q({\bf a})\leq N_p({\bf a})\leq \frac{q-1}{p-1}N_q({\bf a})$$ and $$N_q^f({\bf a})\leq N_p^f({\bf a})\leq \frac{q-1}{p-1}N_q^f({\bf a}).$$
\end{prop}
\begin{proof}
First we handle reverse-reading complexity. Without loss of generality, assume that the output alphabet $\Delta$ of the DFAO that produces ${\bf a}$ is a subset of $\F_{p^N}$ for some $N$ with $\F_q\subseteq\F_{p^N}$. The lower bound on $N_p({\bf a})$ is clear from the fact that $S_q(y)\subseteq S_p(y)$ for $y=\sum {\bf a}(n)x^n$.

For the upper bound, let $M_p$ be a minimal reverse-reading $p$-DFAO that outputs ${\bf a}$. Observe that $M_p$ contains the $q$-DFAO $M_q$ (which also outputs ${\bf a}$) as a ``sub-DFAO'', where the transitions in $M_q$ are achieved by following $r$-fold transitions inside $M_p$. So the states of $M_p$ that are not in $M_q$ comprise at most one $p$-ary tree of height $r$ rooted at each state of $M_q$. So
$$N_p({\bf a})\leq (1+p+p^2+\dots+p^{r-1})|M_q|=\frac{p^r-1}{p-1}N_q({\bf a}),$$
which yields the claimed inequality.

To pass to forward-reading state complexity, follow the dualizing construction of Proposition \ref{dimension bound} to embed the states of a forward-reading $p$-DFAO in some vector space over $\F_{p^N}$. Then let $S_p^T$ and $S_q^T$ be monoids consisting of the transposes of the operators in $S_p$ and $S_q$. The same arguments as above now apply.
\end{proof}

\section{Curves, the Cartier Operator, and Christol's theorem}\label{Curves}

\subsection{Curves and the Cartier Operator} At this point we recall some standard definitions and terminology from the algebraic geometry of curves. For an introduction to the subject, see \cite[Chapter IV]{Hartshorne}, \cite[Chapter II]{SilvermanEllipticCurves}, or \cite{Stichtenoth}.

Let $k$ be a perfect field of characteristic $p$ and let $X/k$ be a smooth projective algebraic curve. Denote the function field $k(X)$ by $K$. Let $\Omega=\Omega_{K/k}$ be the $K$-vector space of (K\"ahler) differentials of $K/k$, which is one-dimensional.

Let $P$ be a (closed) point of $X$, or equivalently a place of $K$. (Whenever we refer to points of $X$, we will always mean closed points.) Write $v_P(f)$ or $v_P(\omega)$ for the valuation given by the order of vanishing of $f\in K^\times$ or $\omega\in\Omega\setminus\{0\}$ at $P$. The valuation ring $\O_P$ is defined to be
$$\O_P=\{f\in K^\times : v_P(f)\geq 0\}\cup\{0\}$$
with maximal ideal
$$\frak{m}_P=\{f\in K^\times : v_P(f)\geq 1\}\cup\{0\}.$$
The degree $\deg(P)$ is $\dim_k\O_P /\frak{m}_P$. Write $\res_P(\omega)$ for the residue of $\omega$ at $P$. 

Let $\Div_k(X)$ denote the group of $k$-rational divisors of $X$. If $f\in K^\times$, define 
\begin{align*}
(f)_0 & =\sum_{v_P(f)>0}v_P(f) P,\\
(f)_\infty & =\sum_{v_P(f)<0} -v_P(f) P,\text{ and}\\
(f) & = (f)_0 - (f)_\infty.
\end{align*}
For $D=\sum_{P} n_P P\in\Div_k(X)$, write $D\geq 0$ if $D$ is effective, that is, if $n_P \geq 0$ for all $P$. Define
\begin{equation*}
\mathcal{L}(D)=\{f\in K^\times:(f)+D\geq 0\}\cup\{0\}
\end{equation*}
and
\begin{equation*}
\Omega(D)=\{\omega\in\Omega\setminus\{0\}:(\omega)+D\geq 0\}\cup\{0\}.
\end{equation*}
By the Riemann-Roch theorem, 
\begin{equation*}
\dim_k \Omega(D)=\dim_k\mathcal{L}(-D)+\deg(D)+g-1.
\end{equation*} 
If $D$ is effective, then $\mathcal{L}(-D)=\{0\}$ and 
\begin{equation*}\label{eqn: differential Riemann-Roch}
\dim_k \Omega(D)=\deg(D)+g-1.
\end{equation*}
For an effective divisor $D$, it will be convenient to introduce the nonstandard notation $\sqrt{D}$ for the ``radical" of $D$, that is,
\begin{equation*}
\sqrt{D} = \sum_{v_P(D)>0} P.
\end{equation*}

Let $x\in K$ be a separating variable ($x\notin K^p$, equivalently $dx\neq 0$). For such an $x$, there is some point $P$ of $X$ such that $v_P(x)$ is not divisible by $p$. By an easy argument using valuations at $P$, the powers $1,x,x^2,\dots,x^{p-1}$ are linearly independent over $K^p$. As $[K:K^p]=p$ by standard facts about purely inseparable extensions \cite[Prop 3.10.2]{Stichtenoth}, the set $\{1,x,\dots,x^{p-1}\}$ forms a basis of $K$ over $K^p$. Thus, any $\omega\in\Omega$ can be written as
\begin{equation}
\omega = \left(u_0^p + u_1^px + \dots + u_{p-1}^p x^{p-1}\right)\,dx
\end{equation}
for unique $u_0,\dots,u_{p-1}\in K$. Define a map $\mathcal{C}:\Omega\to\Omega$ by \begin{equation}\mathcal{C}(\omega)=u_{p-1}\,dx. \end{equation} It is true, but far from obvious, that $\mathcal{C}$ does not depend on the choice of $x$ \cite[p.\ 183]{Stichtenoth}. The operator $\mathcal{C}$ is an $\F_p$-linear endomorphism of $\Omega$ known as the Cartier operator. This operator is of great importance in characteristic-$p$ algebraic geometry. It can be extended in a natural way to $r$-forms of higher-dimensional varieties for any $r$, though we do not need this for our purposes (see for example \cite{Cartier} and \cite{Serre}).

It follows from the definition of $\mathcal{C}$ that for any $\omega\in\Omega$,
\begin{equation}\label{eqn: differential decimation}
\omega=\sum_{i=0}^{p-1} x^i \left(\frac{\mathcal{C}(x^{p-1-i}\omega)}{dx}\right)^p\,dx.
\end{equation}
Comparing equations \ref{eqn: power series decimation} and \ref{eqn: differential decimation} motivates the following definition. For $i\in\{0,1,\dots,p-1\}$, define the \emph{twisted Cartier operator} $\sigma_i:\Omega\to\Omega$ by 
\begin{equation}
\sigma_i(\omega)=\C(x^{p-1-i}\omega).
\end{equation}
 For this to make sense in an arbitrary function field $K$, we need to fix a distinguished separating $x\in K$ in advance (equivalently, a distinguished separable cover $X\to\P^1$). Having done so, if $y\in k((x))\cap K$, it is clear that
\begin{equation}\label{eqn: Cartier and Cartier}
\sigma_i(y\,dx) = \Lambda_i(y)\,dx,
\end{equation}
so the $\sigma_i$ act on differentials just as the $\Lambda_i$ act on series.

\begin{rem}
Equation \ref{eqn: Cartier and Cartier} is true in the differential module of any function field $K$ that contains the Laurent series $y$, as long as $x\in K$ is separating. In particular, we can take $K=k(x)[y]$, as the Laurent series field $k((x))$ is a separable extension of $k(x)$. This proves that if $y$ is algebraic, then the operators $\Lambda_i$ map $k(x)[y]$ into itself. This is not at all obvious from the definition of $\Lambda_i$ as an operator on formal Laurent series.
\end{rem}

We summarize some important properties of $\C$ in the next proposition. These are standard (see e.g. \cite[p.\ 182]{Stichtenoth}), but we sketch proofs for the convenience of the reader.
 
\begin{prop}\label{prop: Cartier properties}
For any $\omega,\omega'\in\Omega$, $f\in K$, and any point $P$ of $X$:
\begin{enumerate}
\item $\mathcal{C}(\omega+\omega')=\C(\omega)+\mathcal{C}(\omega')$.
\item $\C(f^p\omega)=f\C(\omega)$.
\item $\C(\omega)=0$ if and only if $\omega=dg$ for some $g\in K$.
\item If $\omega$ is regular at $P$, then so is $\C(\omega)$.
\item If $\omega$ has a pole at $P$, then $v_P(\C(\omega))\geq \frac{v_P(\omega)+1}{p}-1$, and equality holds if the RHS is an integer. In particular, if $v_P(\omega)=-1$, then $v_P(\C(\omega))=-1$.
\item If $\deg(P)=1$, then $\res_P(\C(\omega))^p=\res_P(\omega)$.

\end{enumerate}
\end{prop}
\begin{proof}

Statements (1) and (2) are immediate from the definition and imply that $\C$ is $\F_p$-linear.

Statement (3) follows from the fact that there is no $g\in K$ such that $dg = x^{p-1}\,dx$. If there were such a $g$, we would have $\frac{dg}{dx}=x^{p-1}$, but this is impossible because the derivative of $x^p$ is zero. For the converse, if $u_{p-1}=0$, set
\begin{equation*}
g=u_0^px +u_1^p\frac{x^2}{2}+\dots+u_{p-2}^p\frac{x^{p-1}}{p-1},
\end{equation*}
and note that no denominator is zero. Then $\omega=dg$.

For statement (4), choose a uniformizer $t$ at $P$, which is necessarily separating (if $t$ is a $p$th power, then $v_P(t)$ is a multiple of $p$ and $t$ cannot be a uniformizer at $P$). Let $\omega=f\,dt$. If $\omega$ is regular at $P$, then so is $f$, because $v_P(dt)=0$. So $f$ can be written as a power series in $t$ and $\C(\omega)=\Lambda_{p-1}(f)\,dt$. The series $\Lambda_{p-1}(f)$ is regular at $P$, so $\C(\omega)$ is also. Statements (5) and (6) follow similarly by writing $f$ as a Laurent series in $t$.

\end{proof}

\subsection{Christol's Theorem and Complexity Bounds}

At this point we fix a prime $p$ and a prime power $q=p^r$. Let $y=\sum_{n=0}^\infty {\bf a}(n)x^n\in \F_q[[x]]$ be algebraic of degree $d$, height $h$, and genus $g$. Let $X$ be the normalization of the projective closure of the affine curve defined by the minimal polynomial of $y$ (after clearing denominators). Set $K=\F_q(X)$ and $\Omega=\Omega_{K/\F_q}$. 

Define $\S_q$ to be the monoid generated by all $r$-fold compositions of the operators $\{\sigma_0,\dots,\sigma_{p-1}\}$. In particular, $\S_p=\langle\sigma_0,\dots,\sigma_{p-1}\rangle$. We write $\S_q(\omega)$ for the orbit of $\omega$ under $\S_q$. Note that $dx\neq 0$ because $\F_q(x)\subseteq K\subseteq \F_q((x))$, so $K/\F_q(x)$ is separable (the Laurent series field is a separable extension of the rational function field). Therefore Equation \ref{eqn: Cartier and Cartier} holds, so the orbits $\S_q(y\,dx)$ and $S_q(y)$ are in bijection. So if $\S_q(y\,dx)$ is finite, then ${\bf a}$ is $p$-automatic, and $|\S_q(y\,dx)|=N_q({\bf a})$.

We now present the proof of the ``algebraic implies automatic" direction of Christol's theorem due to David Speyer \cite{Speyer}. As indicated above, the crux of the argument is to show that the orbit $\S_q(y\,dx)$ is finite. By Proposition \ref{p vs q}, it loses essentially nothing to replace $\S_q(y\,dx)$ with the larger orbit $\S_p(y\,dx)$. 

\begin{prop}[Speyer]\label{new Christol proof}
The sequence ${\bf a}$ is $q$-automatic.
\end{prop}

\begin{proof}

Let $P$ be a point of $X$. By Proposition \ref{prop: Cartier properties}, if neither $x$ nor $\omega$ has a pole at $P$, then $\sigma_i(\omega)=\C(x^{p-i-1}\omega)$ does not have a pole at $P$ for any $i\in\{0,\dots,p-1\}$. Therefore, the only places where elements of $\S_p(y\,dx)$ can have poles are the finitely many poles of $y\,dx$ and of $x$.

Now assume that $P$ is a pole of $y\,dx$ or of $x$. Let $n=v_P(y\,dx)$ and $m=v_P(x)$. Applying the inequality of Proposition ~\ref{prop: Cartier properties} gives
\begin{equation*}
v_P(\sigma_i(y\,dx))\geq\frac{n+m(p-1-i)+1}{p}-1.
\end{equation*}
The pole of largest order that $\sigma_i(y\,dx)$ could have at $P$ occurs when both $n$ and $m$ are negative. In this case,
\begin{equation*}
v_P(\sigma_i(y\,dx))\geq\frac{n}{p}+\frac{m(p-1)}{p}+\frac{1}{p}-1.
\end{equation*}
Applying the same reasoning again shows that
\begin{equation*}
v_P(\sigma_j\sigma_i(y\,dx))\geq \frac{n}{p^2} + \frac{m(p-1)}{p^2}+ \frac{m(p-1)}{p}+\frac{1}{p^2}-1,
\end{equation*} 
and applying it $k$ times shows that 
\begin{align*}
v_P(\sigma_{i_k}\dots\sigma_{i_1}(y\,dx)) & \geq \frac{n}{p^k}+m(p-1)\left(\frac{1}{p^k}+\frac{1}{p^{k-1}}+\dots+\frac{1}{p}\right)+\frac{1}{p^k}-1\\
& \geq n + \frac{m(p-1)}{p-1}+\frac{1}{p^k}-1\\
& >n+m-1.
\end{align*}
Therefore $v_P(\omega)\geq n+m$ for any $\omega\in \S_p(y\,dx)$. (If one of $\{n,m\}$ is positive, then it follows in the same way that $v_P(\omega)\geq\min\{n,m\}$ instead.)

The differentials in $\S_p(y\,dx)$ have poles at only finitely many places, and the orders of these poles are bounded. So there is a finite-dimensional $\F_q$-vector space that contains $\S_p(y\,dx)$, and in particular $\S_p(y\,dx)$ is finite. 
\end{proof}

The Riemann-Roch bound implicit in Proposition \ref{new Christol proof} gives a complexity bound that is a preliminary version of Theorem \ref{thm: main}. This is Corollary \ref{cor: easy upper bound}, for which it will be convenient to use the language of representations. Let $v=y\,dx$, and let $V$ and $\lambda$ be as in the setup before Proposition \ref{new Christol proof}. Let $\phi:\Sigma_q^*\to\End(V)$ be the unique monoid morphism defined for $c\in\Sigma_q$ by $$\phi(c)=\sigma_{i_r}\sigma_{i_{r-1}}\dots\sigma_{i_1}$$ where $c=i_r p^{r-1}+i_{r-1}p^{r-2}+\dots+i_2p+i_1$ with $0\leq i_1,\dots, i_r\leq p-1$. Then by the power series machinery of Section 2.3 and the bijection between $S_q(y)$ and $\S_q(y\, dx)$, we see that $(V,v,\phi,\lambda)$ gives a $q$-representation of ${\bf a}$.

\begin{rem}
If $D$ is any divisor such that $\S_q(y\,dx)\subseteq \Omega(D)$, then we can identify $\lambda$ with an element of $H^1(X,\O_X(-D))$. This is because of the natural duality isomorphism
$$H^1(X,\O_X(-D))^*\simeq H^0(X,\Omega^1_X(D)),$$ which is the classical statement of Serre duality for curves \cite[Chapter III.7]{Hartshorne}.
(The global sections of $\Omega^1_X(D)$ are exactly what we have called $\Omega(D)$.) In fact, $\lambda$ has an explicit realization as a repartition (adele). See e.g. \cite[Chapter 1.5]{Stichtenoth} or \cite[p.\ 37]{Serre}. Moreover, the Cartier operator on $\Omega(D)$ is the transpose (or adjoint) of the Frobenius operator on $H^1(X,\O_X(-D))$.
\end{rem}

\begin{cor}\label{cor: easy upper bound}
$\max(N_q^f({\bf a}),N_q({\bf a}))\leq q^{h+3d+g-1}$
\end{cor}
\begin{proof}
The bounds on the orders of poles in the proof of Proposition \ref{new Christol proof} show that $$\S_q(y\,dx)\subseteq \Omega((y\,dx)_\infty+(x)_\infty),$$ so we may take $V=\Omega((y\,dx)_\infty+(x)_\infty)$ in the $q$-representation of ${\bf a}$. By Proposition \ref{dimension bound}, both $N_q^f({\bf a})$ and $N_q({\bf a})$ are at most $|V|=q^{\dim_{\F_q} V}.$ We have\begin{equation*}
\dim_{\F_q} V =\deg((y\,dx)_\infty+(x)_\infty)+g-1\leq \deg((y\,dx)_\infty)+\deg((x)_\infty)+g-1.
\end{equation*}
Let $\pi_x,\pi_y:X\to \P^1$ be the projection maps from $X$ onto the $x$- and $y$-coordinates. We have $\deg((x)_\infty)=\deg(\pi_x)=d$ and $\deg((y)_\infty)=\deg(\pi_y)=h$. (The easiest way to see this is by looking at the function field inclusions $\pi_x^*:\F_q(x)\hookrightarrow K$ and $\pi_y^*:\F_q(y)\hookrightarrow K$. That is, $d=[K:\F_q(x)]$ and $h=[K:\F_q(y)]$.)

The poles of $dx$ occur at points which are poles of $x$, and the order of a pole of $dx$ at $P$ can be at most one more than the order of the pole of $x$ at $P$. So $\deg((dx)_\infty)$ is maximized when the poles of $x$ are all simple, in which case $\deg((dx)_\infty)= 2\deg((x)_\infty)=2d$. The fact that $\deg((y\,dx)_\infty)\leq \deg((y)_\infty)+\deg((dx)_\infty)$ gives the upper bound.
\end{proof}
The bound in Corollary \ref{cor: easy upper bound} is superseded by Theorem \ref{thm: main} for large values of $h$, $d$, and $g$. However, it is simple to prove and is already much better than the previous bounds derived from Ore's Lemma.

We aim to prove Theorem \ref{thm: main} by bounding the size of the orbit $\S_q(y\,dx)$. As in the proof of Proposition \ref{new Christol proof}, it will be easier to deal with the larger orbit $\S_p(y\,dx)$. By Proposition \ref{p vs q}, this creates no essential difference in the size of the orbit. To streamline the exposition, we establish some preliminary lemmas. Lemmas \ref{lem: fast behavior} and \ref{prop: eventual behavior} determine the ``eventual behavior" of $y\,dx$ under $\S_p$. The main difficulty is in handling the orbit of $y\, dx$ under the operator $\sigma_0$ (this is related to the special role that $0$ plays in non-uniqueness of base expansions). Recall that $\sqrt{D}$ is the sum of the points in the support of the divisor $D$, neglecting multiplicities.

\begin{lem}\label{lem: fast behavior}
Let $V=\Omega((y)_\infty+(x)_\infty+\sqrt{(x)_\infty})$ and $W=\Omega((y)_\infty+(x)_\infty)$. 
Then for any $i\in\{1,\dots,p-1\}$, $\sigma_i(V)\subseteq W$, and for any $i\in\{0,\dots,p-1\}$, $\sigma_i(W)\subseteq W$.
\end{lem}

\begin{proof}
Let $\omega\in V$. Then for any point $P$, $v_P(\omega)\geq -v_P((y)_\infty)-2v_P((x)_\infty)$. By Proposition \ref{prop: Cartier properties}, 
\begin{align*}
v_P(\sigma_i(y\,dx))=v_P(\C(x^{p-1-i}y\,dx)) & \geq \frac{-v_P((y)_\infty)+(-p-1+i)v_P((x)_\infty)+1}{p} -1\\
& \geq\frac{-v_P((y)_\infty)+1}{p}-\frac{pv_P((x)_\infty)}{p}-1\\
& \geq -v_P((y)_\infty)-v_P((x)_\infty),
\end{align*}
where we have used that $i\geq 1$. A similar calculation shows that $\sigma_i(W)\subseteq W$ for any $i$.
\end{proof}

\begin{lem}\label{prop: eventual behavior}
Let $T$ be the maximum order of any pole of $y$ or zero of $x$. Then 
$$\sigma_0^\ell(y\,dx)\in\Omega(\sqrt{(y)_\infty}-(x_0)+\sqrt{(x)_0}+ (x)_\infty+\sqrt{(x)_\infty})$$ for $\ell\geq \lceil\log_p(T)\rceil$.
\end{lem}

\begin{proof}
As in the proof of Proposition \ref{new Christol proof}, any $\omega\in \S_p(y\,dx)$ can have poles only at the poles of $y\,dx$ or of $x$. Writing locally in Laurent series expansions shows that the poles of $y\,dx$ are all either poles of $y$ or poles of $x$, and in fact $y\, dx\in \Omega((y)_\infty+(x)_\infty+\sqrt{(x)_\infty})$. For any $\omega\in\Omega$, we compute
\begin{equation*}
\sigma_0(\omega)=\C\left(\frac{x^p\omega}{x}\right)=x\C\left(\frac{\omega}{x}\right).
\end{equation*}
and therefore
\begin{equation*}
\sigma_0^n(\omega)=x\C^n\left(\frac{\omega}{x}\right).
\end{equation*}
for every $n\geq 1$.

Let $\alpha=\frac{y\, dx}{x}$. We have $\alpha \in\Omega((y)_\infty+(x)_0+\sqrt{(x)_\infty})$. As $y$ is a power series in $x$, it must be the case that $y$ is a regular function at every zero of $x$, so no point can be both a pole of $y$ and a zero of $x$. Let $P$ be a point that is either a pole of $y$ or a zero of $x$. Thus $v_P(\alpha)\geq -T$. Repeatedly applying Proposition \ref{prop: Cartier properties}, we see that $v_P(\C^\ell(\alpha))\geq -1$ for $\ell\geq\log_p(T)$. Therefore $\C^\ell(\alpha)\in \Omega(\sqrt{(y)_\infty}+\sqrt{(x)_0}+\sqrt{(x)_\infty})$. As $\sigma_0^\ell(y\, dx)=x\C^\ell(\alpha)$, we conclude
$$\sigma_0^\ell(y\, dx)\in \Omega(\sqrt{(y)_\infty}+\sqrt{(x)_0}+\sqrt{(x)_\infty}-(x)_0+(x)_\infty)$$
as claimed.
\end{proof}

The next lemma handles the repeated action of $\C$ on differentials with simple poles.

\begin{lem}\label{simple poles}
Suppose $\omega\in K$ has simple poles at points of degrees $e_1,e_2,\dots,e_n$. Let $m$ be the LCM of $e_1,\dots,e_n$. Then $\C^{rm}(\omega)-\omega$ is holomorphic (recall $q=p^r$).
\end{lem}

\begin{proof}
Let $X'$ be the base change $X'=X\otimes_{\F_q}\F_{q^m}$ with base change morphism $\phi:X'\to X$. Let $K'=\phi^* K$, which is the constant field extension $\F_{q^m}K$. Each place $P$ of $K$ which is a pole of $\omega$ splits completely in the extension to $K'$ (for example, by \cite[Thm 3.6.3 g]{Stichtenoth} each place $P'$ lying over $P$ has residue field equal to $\F_{q^m}$). Therefore the pullback $\phi^*\omega$ has simple poles at places of degree 1. So $\C(\phi^*\omega)$ has simple poles at the same places as $\phi^*\omega$. At each of these places $P'$, we compute
\begin{equation*}
\res_{P'}(\C^{rm}(\phi^*\omega))=\res_{P'}(\phi^*\omega)^{(1/p)^{rm}}=\res_{P'}(\phi^*\omega)^{q^{-m}}=\res_{P'}(\phi^*\omega)
\end{equation*}
because the residue lies in $\F_{q^m}$, which is fixed under the $q^m$th power map. So $\C^{rm}(\phi^*\omega)$ has simple poles at the same places as $\phi^*\omega$ with the same residues, and therefore $\C^{rm}(\phi^*\omega)-\phi^*\omega$ is holomorphic. The Cartier operator commutes with pullback, so $$\C^{rm}(\phi^*\omega)-\phi^*\omega=\phi^*(\C^{rm}(\omega)-\omega)$$ and we conclude that $\C^{rm}(\omega)-\omega$ is also holomorphic. 
\end{proof}

Using the preceding lemmas, we now prove Theorem \ref{thm: main}.

\begin{proof}[Proof of Theorem \ref{thm: main}]
Let $V=\Omega((y)_\infty+(x)_\infty+\sqrt{(x)_\infty})$ and $W=\Omega((y)_\infty)+(x)_\infty)$. By Lemma \ref{lem: fast behavior}, $\sigma_i(y\, dx)\in W$ for every $i>0$, and $W$ is $\sigma_i$-invariant for every $i$. So we have
$$|N_q({\bf a})|=|\S_q(y\, dx)|\leq |\S_p(y\, dx)|\leq 1 +  |\{\sigma_0^n(y\, dx ):n\geq 1\}| + |W|.$$
By Riemann-Roch, $\dim_{\F_q} W = \deg((y)_\infty+(x)_\infty)+g-1\leq h+d+g-1$. The remainder of the proof will handle the orbit of $y\,dx$ under $\sigma_0$.

Let $T$ be the maximum order of any pole of $y$ or zero of $x$. Let $D=\sqrt{(y)_\infty}+\sqrt{(x)_0}+\sqrt{(x)_\infty}$. By Lemma \ref{prop: eventual behavior}, for $n\geq\lceil\log_p(T)\rceil$ we have $\sigma_0^n(y\, dx)\in \Omega(D-(x))$. Let $\alpha=x^{-1}\sigma_0^{\lceil\log_p(T)\rceil}(y\,dx)$. So $\alpha\in\Omega(D)$, that is, $\alpha$ has simple poles at points that are either poles of $y$, poles of $x$, or zeroes of $x$. We have seen that
$x\C^n(\alpha)=\sigma_0^n(x\alpha)$. It follows that
$$|\{\C^n(\alpha):n\geq 0\}| = |\{\sigma^n(y\, dx): n\geq \log_p(T) \}|.$$
Let $m$ be the LCM of the degrees of the points at which $\alpha$ has a pole. By Lemma \ref{simple poles}, $\C^{rm}(\alpha)-\alpha$ is holomorphic. The space of holomorphic differentials is invariant under $\C$, so the orbit of $\alpha$ under $\C$ is contained in the set $$\{\C^k(\alpha)+\eta:0\leq k<rm\text{ and }\eta\in\Omega(0)\}.$$ This set has size at most $rm|\Omega(0)|=rmq^g$. Thus 
\begin{equation*}
|\{\sigma_0^n(\omega):n\geq 1\}|\leq \lceil\log_p(T)\rceil + rmq^g.
\end{equation*}
We now need to estimate $m$.

Let $L(n)$ be Landau's function, that is, the largest LCM of all partitions of $n$, or equivalently the maximum order of an element in the symmetric group of order $n$. Recall from the proof of Corollary \ref{cor: easy upper bound} that $\deg((y)_\infty)=h$ and $\deg((x)_\infty)=d$. We have
\begin{align*}
\sum_{v_P(y)<0}\deg(P) \leq \sum_{v_P(y)<0}-v_P(y)\deg(P)=\deg((y)_\infty)= h,
\end{align*}
and it follows in the same way that $\sum_{v_P(x)<0}\deg(P)\leq d$ and $\sum_{v_P(x)>0}\deg(P)\leq d$. 
Therefore $m\leq L(h)L(d)^2$. It is clear that $L(a)L(b)\leq L(a+b)$ for all $a$ and $b$, so $m\leq L(h+2d)$. So
\begin{equation*}
|\{\sigma_0^n(y\,dx):n\geq 1\}|\leq \lceil\log_p(T)\rceil+rL(h+2d)q^g.
\end{equation*}
Therefore
$|\S_p(y\, dx)|\leq  1 + \lceil\log_p(T)\rceil+rL(h+2d)q^g + q^{h+d+g-1}$.

It remains to show that the quantity
$$\frac{1 + \lceil\log_p(T)\rceil+rL(h+2d)q^g}{q^{h+d+g-1}}$$
decays to zero as any of $q,h,d,g$ grow to $\infty$. This follows easily from the fact that $g\geq 0$ and $h+d\geq 2$ for any algebraic curve, the simple bound on Landau's function 
\begin{equation*}
L(n)\leq \exp\left((1+o(1))\sqrt{n\log n}\right)
\end{equation*}
from \cite{Landau}, and the fact that $T\leq \max(h,d)$.
\end{proof}

The forward-reading complexity bound of Theorem \ref{thm: main2} follows as an easy corollary.
\begin{proof}[Proof of Theorem \ref{thm: main2}]
Let $V=\Omega((y)_\infty+(x)_\infty+\sqrt{(x)_\infty})$ and let $\lambda\in V^*$ be the linear functional that maps $\omega$ to the constant term of the power series $\frac{\omega}{dx}$. We have $\dim_{\F_q} V\leq h+2d+g-1$, so $$|N_q^f({\bf a})|=|\S_q^T(\lambda)|\leq |V^*|\leq q^{h+2d+g-1}$$
by Proposition \ref{dimension bound}.
\end{proof}

We now show that Theorem \ref{thm: main} is qualitatively sharp for the power series expansions of rational functions, that is, it is sharp if we replace the ``error term" $o(1)$ by zero.

\begin{prop}\label{rational sharp}
For every prime power $q$ and every positive integer $h\geq 1$, there exists $y=\sum_{n=0}^\infty {\bf a}(n)x^n\in\F_q[[x]]$ with $\deg(y)=1$ and $h(y)=h$ (and therefore $g=0$) such that $N_q({\bf a})\geq q^h$.
\end{prop}
\begin{proof}
Let $f=x^h+c_{h-1}x^{h-1}+\dots+c_1x+c_0 \in\F_q[x]$ be any primitive polynomial, that is, such that a root of $f$ generates $\F_{q^h}^\times$. Let $y=f^{-1}-1\in\F_q[[x]]$. The coefficient sequence ${\bf a}$ of $y$ satisfies the linear recurrence relation
$$c_0{\bf a}(n)+c_{1}{\bf a}(n-1)+\dots+c_{h-1}{\bf a}(n-h+1)+{\bf a}(n-h)=0$$ and ${\bf a}$ is eventually periodic with minimal period $q^h-1$ \cite[Thm 6.28]{LidlNiederreiter}. We have $\deg(y)=1$ and $h(y)=h$, so the curve $X$ is $\P^1$, with $K=\F_q(x)$ and $g=0$. 

Note that $(y)_\infty=(f)_0$ is a single point of degree $h$. Let $P_\infty$ be the pole of $x$, which is distinct from $(f)_0$. We compute $$(f^{-1}\,dx)=hP_\infty-(f)_0-2P_\infty=(h-2)P_\infty-(f)_0,$$ so $f^{-1}\,dx$ has at most two poles: a simple pole at $(f)_0$ of degree $h$, and if $h=1$, a simple pole at $P_\infty$ of degree 1. By Lemma \ref{simple poles}, $\C^{rh}(y\,dx)-f^{-1}\,dx=\C^{rh}(f^{-1}\,dx)-f^{-1}\,dx$ is holomorphic (note $\C(dx)=0$ by Proposition \ref{prop: Cartier properties}). As $X$ has genus 0, it carries no nonzero holomorphic differentials, so $\C^{rh}(y\,dx)=f^{-1}\,dx$.

Let ${\bf b}$ be the coefficient sequence of $f^{-1}$. The sequence ${\bf b}$ satisfies a linear recurrence relation of degree $h$ and has period $q^h-1$, so it must be the case that all possible strings of $h$ elements in $\F_q$ except for the string $(0,\dots,0)$ occur in ${\bf b}$ within the first $q^h-1$ terms. For each $0\leq c\leq q-1$, a certain $r$-fold composition of $\Lambda_i$ operators $s\in S_q$ gives $s(f^{-1})=\sum_{n=0}^\infty {\bf b}(qn+c)x^n$, so 
$$\Lambda_0^{rh}s(f^{-1})=\sum_{n=0}^\infty {\bf b}(q^hn+c)x^n=\sum_{n=c}^\infty {\bf b}(n)x^n$$ 
by the periodicity of ${\bf b}$. So there are at least $q^h-1$ distinct power series in $S_q(f^{-1})$.

Let $V=\Omega((f)_0+P_\infty)$. We have $(f^{-1}\,dx)\in V$, and $\dim_{\F_q} V = h$ by Riemann-Roch. A calculation with properties of $\C$ and orders of poles shows that $V$ is $\sigma_i$-invariant for any $i\in\{0,\dots,p-1\}$, so $\S_q(f^{-1}\,dx)\subseteq V$. The counting argument from the previous paragraph shows that the orbit $\S_q(f^{-1}\,dx)$ comprises all nonzero elements of $V$. (In fact, it is not hard to show that $\sigma_i |_V$ is invertible for each $i$, so the action of $\S_p$ on $V$ is a group action with precisely two orbits: $\{0\}$ and $V\setminus\{0\}$.) Note that $y\, dx\notin V$, for if $y\,dx$ were in $V$, then $dx$ would be also, but $(dx)=-2P_\infty$. So $y\notin S_q(f^{-1})$, which establishes the lower bound $|S_q(y)|=N_q({\bf a})\geq q^h-1+1=q^h$.



\end{proof}

\subsection{Variation mod primes}\label{Variation mod p}

Let $K$ be a number field and let $y=\sum_{n=0}^\infty {\bf a}(n) x^n \in K[[x]]$. If the prime $\p$ of $K$ is such that $v_\p({\bf a}(n))\geq 0$ for all $n$, let ${\bf a}_\p$ denote the reduction of ${\bf a}$ mod $\p$, and let $y_\p=\sum_{n=0}^\infty {\bf a}_\p(n) x^n$ be the reduced power series with coefficients in the residue field $k(\p)$.

Suppose $y$ is algebraic over $K(x)$. By an old theorem of Eisenstein, there are only finitely many primes $\p$ such that $v_\p({\bf a}(n))<0$ for some $n$ (\cite[pp.\ 765-767]{Eisenstein}, see also \cite{Schmidt}). So the sequence ${\bf a}_\p$ is defined for all but finitely many $\p$, and by Christol's theorem it is $|k(\p)|$-automatic (it is an easy observation that the reduction mod $\p$ of an algebraic function is algebraic). An extension of our main question is how the algebraic nature of $y$ affects the complexity $N_{|k(\p)|}({\bf a}_\p)$ as the prime $\p$ varies. Theorem \ref{no growth} answers this question in the case that the complexities are bounded at all primes; in this case $y$ must have a very special form. Note that we do not need to assume that $y$ is algebraic in the statement of the theorem. To simplify notation, we will write $N_\p({\bf a_\p})$ in place of $N_{|k(\p)|}({\bf a}_\p)$. 

\begin{thm}\label{no growth}
Let $y=\sum_{n=0}^\infty {\bf a}(n)x^n\in K[[x]]$. Then $N_\p({\bf a}_\p)$ and $N_\p^f({\bf a}_\p)$ are bounded independently of $\p$ if and only if $y$ is a rational function with at worst simple poles that occur at roots of unity (except possibly for a pole at $\infty$, which may be of any order).
\end{thm}
\begin{proof}
It suffices to prove the theorem for reverse-reading complexity, as $N_\p^f({\bf a})\leq p^{N_\p({\bf a})}$ for any sequence ${\bf a}$ by Proposition \ref{dimension bound}. Assume that $N_\p({\bf a})$ is uniformly bounded for all $\p$ (such that it is defined). Then the coefficient sequence ${\bf a}$ assumes a bounded number of values under reduction mod $\p$ regardless of $\p$, and so ${\bf a}$ assumes finitely many values in $K$. Let $A$ be this finite subset of $K$.

Choose two primes $\p$ and $\q$ such that $|k(\p)|$, $|k(\q)|$ are multiplicatively independent integers (i.e. char $k(\p)\neq$ char $k(\q)$) and all elements of $A$ are distinct both mod $\p$ and mod $\q$ (this is possible because only finitely many primes divide distances between distinct elements of $A$). The reduced power series $y_\p\in k(\p)[[x]]$ and ${y}_\q\in k(\q)[[x]]$ are both algebraic by Christol's theorem. So there exist injections $i_\p:A\hookrightarrow k(\p)$ and  $i_\q:A\hookrightarrow k(\q)$ such that the sequence ${\bf b}(n)= i_\p({\bf a}(n))$ is $|k(\p)|$-automatic and the sequence ${\bf c}(n)=i_\q({\bf a}(n))$ is $|k(\q)|$-automatic. Therefore ${\bf a}$ is both $|k(\p)|$-automatic and $|k(\q)|$-automatic. By Cobham's theorem \cite[Thm 11.2.2]{AS}, ${\bf a}$ is an eventually periodic sequence of some period $m$, so $y$ is a rational function of the form $y=\frac{f}{1-x^m}$ for some polynomial $f$. So the (finite) poles of $y$ are simple and occur at roots of unity.

Conversely, assume that the (finite) poles of $y$ are simple and occur at roots of unity. Therefore $y=\frac{f}{1-x^m}$ for some $m$ and $f\in K[x]$, and the coefficient sequence ${\bf a}$ is eventually periodic of (possibly non-minimal) period $m$, that is, there is some $c$ such that ${\bf a}(n+m)={\bf a}(n)$ for all $n>c$. In particular, ${\bf a}$ assumes finitely many values. An easy decimation argument now shows that $N_\p({\bf a})$ is uniformly bounded for all primes. Suppose $|k(\p)|=p^r$ and assume that $p^r>c$ (which excludes only finitely many $\p$). For any $i\geq 1$ and $j\in\{0,\dots,p^{ri}-1\}$, the subsequences ${\bf a}(p^{ri}n+j)$ are periodic of period $m$ beginning with the second term, and they assume the same finite set of values as ${\bf a}$. There are clearly only finitely many sequences that fit this description. So the size of the $p^r$-kernel of ${\bf a}$, and therefore $N_\p({\bf a})$, is bounded independently of $p$.

\end{proof}

\section{Examples}\label{Examples}

We give three detailed examples of computing the state complexity of an automatic sequence. Examples \ref{binomial series} and \ref{elliptic example} in particular show the usefulness of the algebro-geometric approach.

\begin{exa}
$y=\dfrac{1}{1-2x}$.
\end{exa}

\begin{figure}[ht]
\caption{$7$-DFAO generating powers of 2 mod $7$}
\includegraphics[scale=.75]{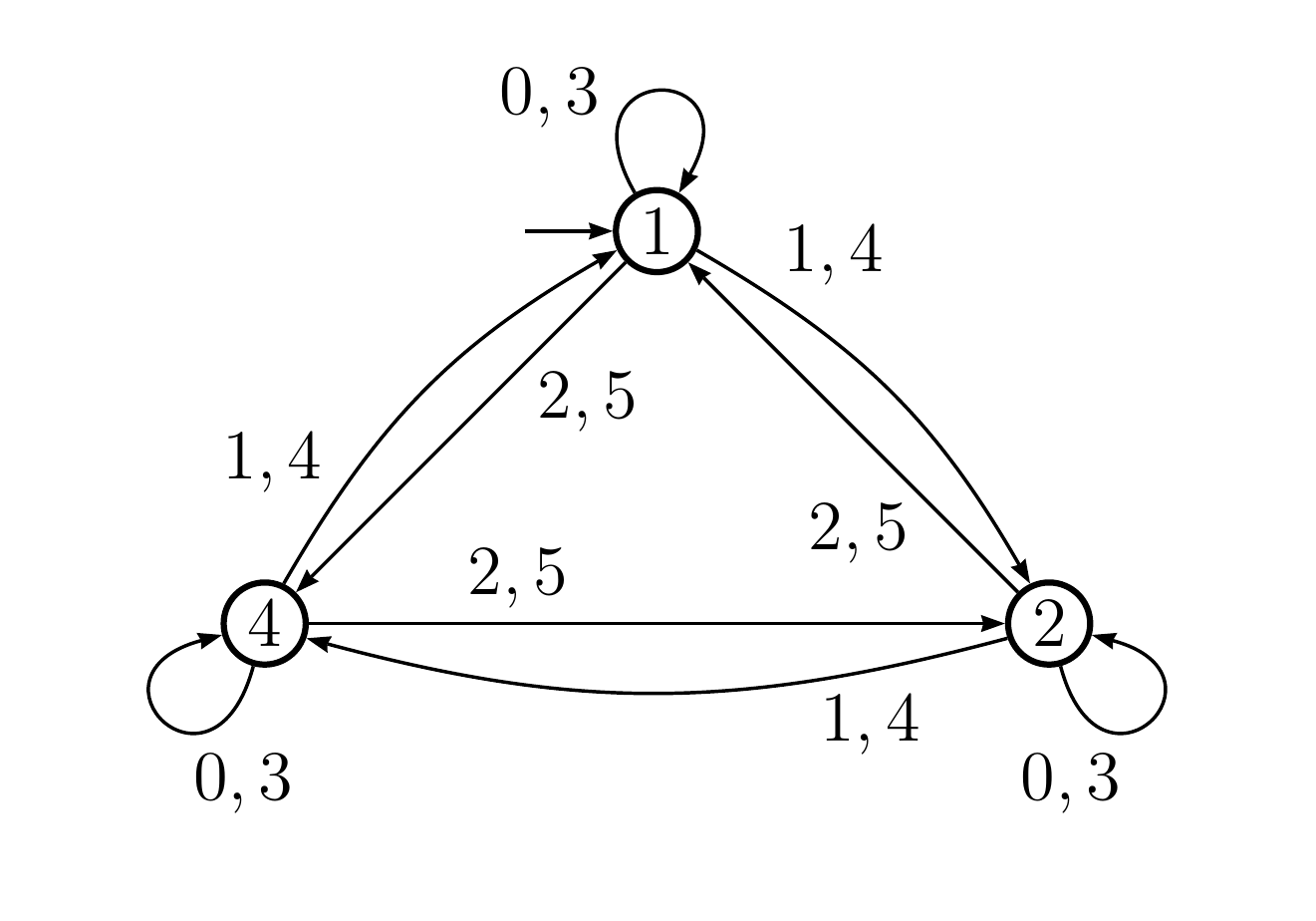}
\label{Powers of 2}
\end{figure}

Let $p$ be odd. Let ${\bf a}(n)=2^n$ mod $p$ and $y=\sum_{n=0}^\infty{\bf a}(n)x^n\in\F_p[[x]]$. We have $y(1-2x)=1$, so $y$ has degree 1, height 1, and genus 0. By Theorem \ref{thm: main}, $N_p({\bf a})\leq (1+o_p(1))p$. We compute $\Lambda_i(y)=2^i y$, so $S_p(y)=\{2^i y:i\geq 0\}$, and $N_p({\bf a})=\ord_p(2)$. So in fact $$\lceil\log_2(p)\rceil\leq N_p({\bf a})\leq p-1.$$ If there are infinitely many Mersenne primes, the lower bound is sharp infinitely often, and if Artin's conjecture is true, the upper bound is sharp infinitely often.

The sequence ${\bf a}$ has a one-dimensional $p$-representation where $\phi(i):v\mapsto 2^iv$. Each $\phi(i)$ can be written as a (symmetric) $1\times 1$ matrix, so the $p$-antirepresentation on $V^*$ is the same as the original representation, and $N_p^f({\bf a})=N_p({\bf a})$. From the automata point of view, this is the obvious fact that the same DFAO outputs ${\bf a}$ in both the forward-reading and reverse-reading conventions. The transition diagram of the DFAO for $p=7$ is given in Figure \ref{Powers of 2}.

\begin{exa}\label{binomial series}
$y=\dfrac{1}{\sqrt{1-4x}}$.
\end{exa}

\begin{figure}[ht]
\caption{5-DFAO generating central binomial coefficients mod 5}
\includegraphics[scale=.75]{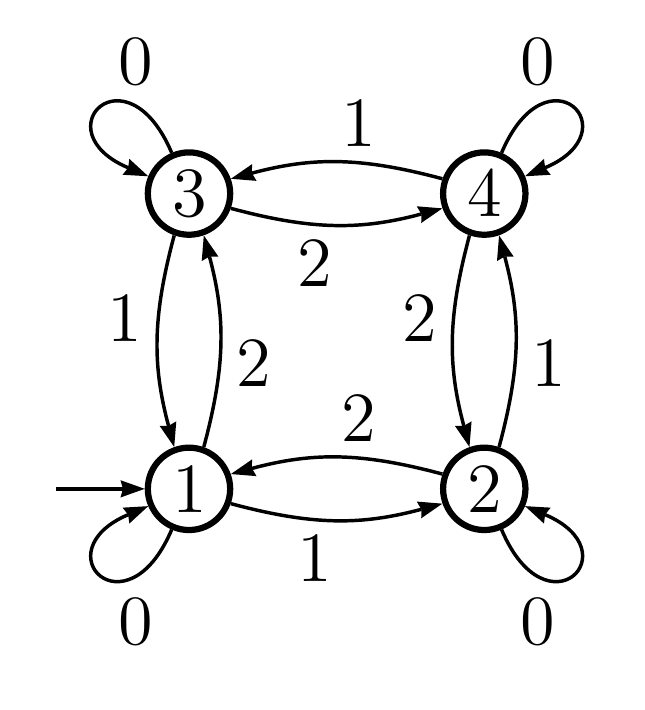}
\label{BinomialSeries}
\end{figure}

Let $p$ be an odd prime. Let ${\bf a}(n)$ be the central binomial coefficient ${2n\choose n}$ reduced mod $p$ and let $y=\sum_{n=0}^\infty{\bf a}(n)x^n\in\F_p[[x]]$. From Newton's formula for the binomial series, we have $y^2(1-4x)=1$. So $y$ has degree 2, height 1, and genus 0, and $N_p({\bf a})\leq (1+o_p(1))p^2$. We show that $N_p({\bf a})=N_p^f({\bf a})=p$. (A calculation verifies that $N_2^f({\bf a})=N_2({\bf a})=2$ also.)

Let $X$ be the curve defined by $y^2(1-4x)=1$. We have
$x=\frac{1}{4}-y^{-2}$, so $X=\P^1(\F_p)$, parametrized by $y$. Let $P_0$ and $P_\infty$ be the zero and pole of $y$, and let $\omega=y\,dx$. A computation gives $dx =-2y^{-3}\,dy$, so
\begin{align*}
(\omega) &= -2(P_0)\text{ and}\\
(x) &= (P_2)+(P_{-2}) - 2(P_0).
\end{align*}
For any $i\in\{0,\dots, p-1\}$,
\begin{equation*}
v_{P_0}(x^{p-1-i}\omega)=(p-1-i)v_{P_0}(x)+v_{P_0}(\omega)>\left(p-1\right)(-2)-2 = -2p.
\end{equation*}
So $v_{P_0}(\sigma_i(\omega))=v_{P_0}(\mathcal{C}(x^{p-1-i}\omega))$ is either $0,-1$, or $-2$, and $P_0$ is the only point at which $\sigma_i(\omega)$ can have a pole. A canonical divisor of $X$ has degree $-2$, so by Riemann-Roch, $\Omega(2P_0)$ is one-dimensional, $\S_p(\omega)\subseteq \F_p\omega$, and $N_p({\bf a})\leq p$.

More explicitly, as  $\S_p(\omega)$ sits in a one-dimensional vector space we have $\sigma_i(\omega)=c_i\omega$ for some $c_i\in\F_p$. As $\sigma_i(\omega)=\Lambda_i(y)\,dx$ and the constant term of $y$ is 1, $c_i$ is equal to the constant term of $\Lambda_i(y)$, which is $\binom{2i}{i}$. So $\sigma_i(\omega)=\binom{2i}{i}\omega$. Equating coefficients gives
\begin{equation*}
\binom{2(pn+i)}{pn+i}\equiv \binom{2i}{i}\binom{2n}{n}\pmod{p}.
\end{equation*}
(Amusingly, this gives a roundabout argument that recovers a special case of the classical theorem of Lucas on binomial coefficients mod $p$.)

To see that $N_p({\bf a})$ is exactly $p$, note that ${\bf a}(1)=2$, and that for any odd prime $q$, ${\bf a}\left(\frac{q+1}{2}\right)$ is the first central binomial coefficient divisible by $q$. This shows that the subgroup of $\F_p^\times$ generated by all nonzero central binomial coefficients mod $p$ contains all primes less than $p$ and therefore is all of $\F_p^\times$, and furthermore ${\bf a}\left(\frac{p+1}{2}\right)=0$. 

As in the previous example, the fact that $\S_p(\omega)$ lies in a one-dimensional vector space verifies that $N_p^f({\bf a})=N_p({\bf a})$ for all $p$. The transition diagrams for the automata that output ${\bf a}$ have nicely symmetric structures. Figure \ref{BinomialSeries} displays the DFAO for $p=5$ -- all undrawn transitions, which are on the inputs 3,4, and 5, go to an undrawn trap state, which outputs zero.

\begin{exa}\label{elliptic example}
$y=\dfrac{1}{\sqrt{1-4x^3}}$.
\end{exa}

\begin{figure}[ht]
\caption{Reverse-reading 5-DFAO generating coefficients of $(1-4x^3)^{-1/2}$ mod 5}
\includegraphics[scale=.75]{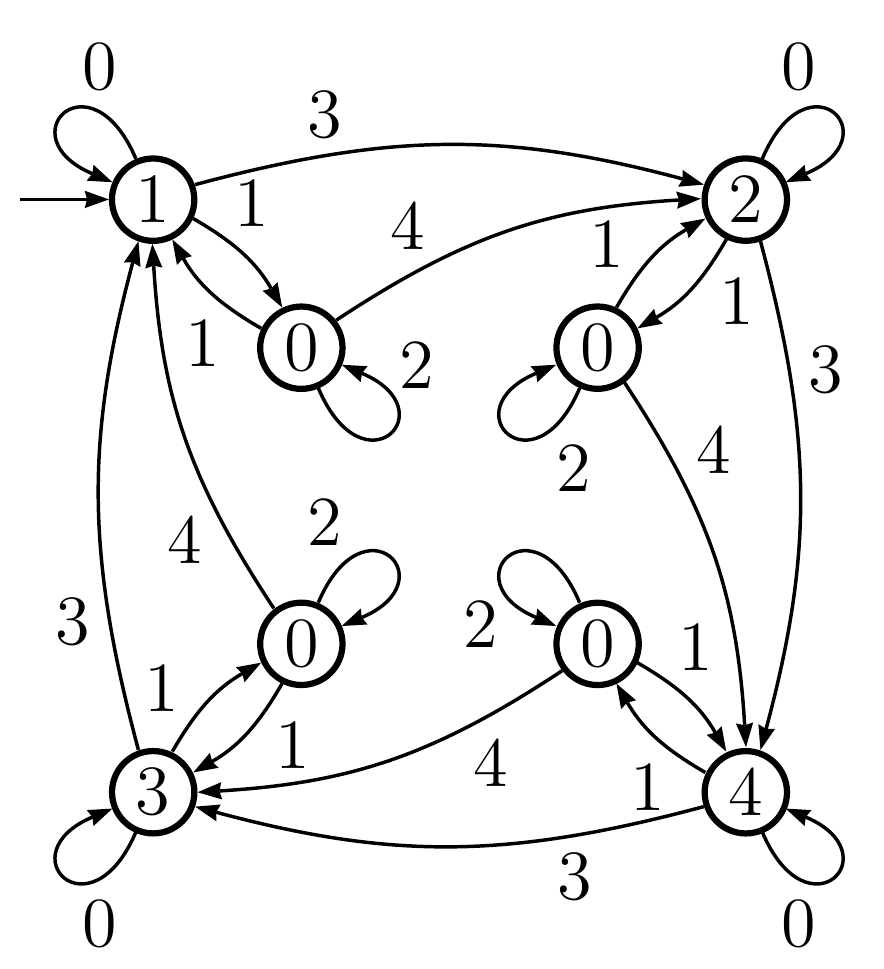}
\label{EllipticBinomial}
\end{figure}

Let $p\notin\{2,3\}$. Let ${\bf a}$ be the coefficient sequence of the series $\frac{1}{\sqrt{1-4x^3}}\in\Q[[x]]$ reduced mod $p$, so that ${\bf a}(3n) = \binom{2n}{n}$ mod $p$ and ${\bf a}(n)=0$ if $n$ is not a multiple of 3. Let $y=\sum_{n=0}^\infty {\bf a}(n)x^n$. We have $y^2(1-4x^3)=1$, so $y$ is of degree 2, height 3, and genus 1, and $N_p({\bf a})\leq (1+o_p(1))p^5$. We show that $N_p({\bf a})=2p-1$.

Let $C$ be the curve defined by $Y^2(Z^3-4W^3)=Z^5$ in $\P^2$. This curve is singular, so define the smooth (elliptic) curve $X$ by $Y^2Z=Z^3-4W^3$. The morphism $\phi:X\to C$ defined in homogeneous coordinates by
$$\phi:[W:Y:Z]\mapsto [WY:Z^2:YZ]$$
gives the normalization of $C$. The forms $[Y:Z]$ and $[W:Z]$ are maps from $C$ to $\P^1$, so we can consider them as elements of $\F_p(C)$. Let $y=\phi^*[Y:Z]$ and $x=\phi^*[W:Z]$. So we have $y^2(1-4x^3)=1$. Let $\omega=y\,dx$ and let $P_\infty$ be the point on $X$ written as $[0:1:0]$ in homogeneous coordinates. The following are easy computations, where $P,Q,$ and $R$ are some points of $X$ that we do not need to compute explicitly:
\begin{align*}
(y) & = 3P_\infty - P - Q - R\\
(x) & = [0:1:1] + [0:-1:1] - 2P_\infty\\
(dx) & = -(y).
\end{align*}
In particular, $(\omega)=0$. Our usual computation for the possible orders of poles of $\sigma_i(\omega)$ shows that $\S_p(\omega)\subseteq\Omega(2P_\infty)$, which has dimension 2 with $\{\omega,x\omega\}$ as a basis. We use properties of $\C$ to compute the action of $\S_p$ on the basis.

As $y^2=\frac{1}{1-4x^3}$, we have 
\begin{align*}
\sigma_i(\omega) & =\C\left(x^{p-i-1}y\,dx\right)=\C\left(x^{p-i-1}\frac{y^p}{y^{p-1}}\,dx\right)=y\C(x^{p-i-1}(1-4x^3)^{\frac{p-1}{2}}\,dx)\\
& =y\C\left(\sum_{k=0}^{\frac{p-1}{2}}{\frac{p-1}{2}\choose k}(-4)^kx^{3k+p-i-1}\,dx\right).
\end{align*}
So $\sigma_i(\omega)$ is nonzero precisely when there is some $0\leq k\leq \frac{p-1}{2}$ with $$3k+p-i-1\equiv p-1\pmod{p},$$ that is, when $3k\equiv i\pmod{p}$ has a solution $k$ with $0\leq k\leq \frac{p-1}{2}$. If there is such a $k$, then it is unique, and $3k-i\leq \frac{3(p-1)}{2}<2p,$ so either $3k=i$, in which case $$\sigma_i(\omega)={\frac{p-1}{2}\choose \frac{i}{3}}(-4)^{\frac{i}{3}}\omega,$$
or $3k=i+p$, in which case $$\sigma_i(\omega)={\frac{p-1}{2}\choose \frac{i+p}{3}}(-4)^{\frac{i+p}{3}}x\omega.$$ This shows that $\S_p(\omega)\subseteq\F_p\omega\cup\F_p x\omega$. Also, the fact that $\sigma_i(\omega)=\Lambda_i(y)\,dx$ proves the identity
$${\frac{p-1}{2}\choose k}(-4)^k\equiv{2k\choose k}\pmod{p}$$
for all $k$.

With this calculation we can explicitly write the restriction of $\sigma_i$ to $\Omega(2P_\infty)$ by computing its action on the basis $\{\omega,x\omega\}$. So far we have only computed the action of $\sigma_i$ on $\omega$, but for $i\geq 1$ we have $\sigma_i(x\omega)=\sigma_{i-1}(\omega)$, and $\sigma_{0}(x\omega)=\C(x^p\omega)=x\C(\omega)=x\sigma_{p-1}(\omega)$. We have
$$\sigma_i(\omega)=
\left\{\begin{array}{ll}
{\frac{2i}{3}\choose\frac{i}{3}}\omega & : i\equiv 0\pmod{3} \\
{\frac{2(i+p)}{3}\choose\frac{i+p}{3}} x\omega & : i\equiv -p\pmod{3} \\
0 & : i\equiv p\pmod{3}
\end{array}\right.
$$
and
$$\sigma_i(x\omega)=
\left\{\begin{array}{ll}
{\frac{2(i-1)}{3}\choose\frac{i-1}{3}}\omega & : i\equiv 1\pmod{3} \\
{\frac{2(i-1+p)}{3}\choose\frac{i-1+p}{3}} x\omega & : i\equiv 1-p\pmod{3} \\
0 & : i\equiv 1+p\pmod{3}
\end{array}\right.
$$
where $0,p,-p$ are distinct mod $p$ because $p>3$. 

(Incidentally, it follows from our computation that $\C(\omega)=0$ if and only if $p\equiv 2\pmod{3}$, which shows that these are precisely the primes for which the elliptic curve $X$ is supersingular, as the classical Hasse invariant is the rank of the restriction of the Cartier operator to the space of holomorphic differentials. See \cite[Section 5.4]{SilvermanEllipticCurves}.)

Examining the binomial coefficients that appear in the formula for $\sigma_i(\omega)$ shows that ${2k\choose k}$ appears as a coefficient on $\omega$ for $0\leq k\leq\lfloor p/3\rfloor$, and as a coefficient on $x\omega$ for $\lfloor p/3\rfloor +1\leq k\leq (2p-1)/3$. As in Example \ref{binomial series}, the values of ${2k\choose k}$ mod $p$ for $0\leq k \leq (p-1)/2$ generate the multiplicative group of $\F_p^\times$, and we have $\sigma_1(x\omega)=\omega$. This is already enough to show that $\S_p(\omega)=\F_p\omega\cup\F_px\omega$, so $N_p({\bf a})=2p-1$.

For $p=5$, the reverse-reading 5-DFAO that outputs ${\bf a}$ is pictured in Figure \ref{EllipticBinomial}. As usual, any undrawn transitions lead to a trap state that outputs $0$.

\bibliographystyle{amsplain}
\bibliography{AutomaticSequencesAndCurves}
\nocite{*}

\end{document}